\documentclass[11pt]{amsart}      
\usepackage{amsmath,amssymb,amsthm,enumerate,enumitem}
\usepackage{hyperref} % links to references

\usepackage{kantlipsum}

\calclayout

\newtheorem{thm}{Theorem}
\newtheorem{thmx}{Theorem}

\newtheorem{cor}[thm]{Corollary}
\newtheorem{lemma}[thm]{Lemma}

\newtheorem{fact}[thm]{Fact}

\theoremstyle{definition}
\newtheorem{defin}[thm]{Definition}

\theoremstyle{remark}
\newtheorem{remark}[thm]{Remark}

\newcommand{\Rea}{\mathbb R}
\newcommand{\Nat}{\mathbb N}
\newcommand{\Rat}{\mathbb Q}

\newcommand{\A}{\mathcal A}
\newcommand{\F}{\mathcal F}

\newcommand{\U}{\mathcal U}
\newcommand{\M}{\mathcal M}

\newcommand{\PP}{\mathcal P}
\newcommand{\SSS}{\mathcal S}

\newcommand{\Id}{\operatorname{Id}}
\def \rng {\operatorname{Rng}}
\def \ker {\operatorname{Ker}}
\newcommand{\Span}{\operatorname{span}}
\newcommand{\closedSpan}{\overline{\operatorname{span}}}

\def \dens {\operatorname{dens}}

\def \Lip {\operatorname{Lip}}

\begin{document}
\title[Note on almost isometric ideals]{Note on almost isometric ideals and local retracts in Banach and metric spaces}

\author[L. Candido]{Leandro Candido}
\author[M. C\' uth]{Marek C\'uth}
\author[O. Smetana]{Ond\v{r}ej Smetana}
\email{leandro.candido@unifesp.br}
\email{cuth@karlin.mff.cuni.cz}
\email{ond.smetana@gmail.com}

\address[L.~Candido]{Universidade Federal de S\~{a}o Paulo - UNIFESP. Instituto de Ci\^{e}ncia e Tecnologia. Departamento de Matematica. S\~{a}o Jos\'e dos Campos - SP, Brasil}

\address[M.~C\' uth, O.~Smetana]{Charles University, Faculty of Mathematics and Physics, Department of Mathematical Analysis, Sokolovsk\'a 83, 186 75 Prague 8, Czech Republic}


\thanks{L. Candido was supported by Funda\c c\~ao de Amparo \`a Pesquisa do Estado de S\~ao Paulo - FAPESP no. 2023/07291-2. M. C\'uth was supported by the GA\v{C}R project 23-04776S.}

\begin{abstract}
We exhibit a new approach to the proofs of the existence of a large family of almost isometric ideals in nonseparable Banach spaces and existence of a large family of almost isometric local retracts in metric spaces. Our approach also implies the existence of a large family of nontrivial projections on every dual of a nonseparable Banach space. We prove three possible formulations of our results are equivalent. Some applications are mentioned which witness the usefulness of our novel approach.
\end{abstract}
\maketitle

One of the main issues when trying to understand the geometry of a given Banach is to understand its nontrivial complemented subspaces (by nontrivial we mean not of finite dimension or codimension). On the one hand we have examples of Banach spaces with no nontrivial complemented subspaces, the first one constructed by Gowers and Maurey for separable spaces \cite{Gow93}, more recent result by Koszmider, Shelah and \'Swi\c{e}tek is that consistently there are such examples of arbitrary density \cite{KSS18}. On the other hand, for many nonseparable Banach spaces we are able to find a very nice structure on nontrivial complemented subspaces with consequences for the structure of the space itself, quite a wide class of spaces with this feature are Banach spaces admitting a projectional skeleton (including e.g. all the reflexive spaces, more generally also all the WCG spaces), where the structure of nontrivial complemented subspaces enables us to prove that those spaces admit Markushevich basis and LUR renorming, see e.g. \cite[Chapter 17]{KKLBook} and \cite{Kal20}.

It is also well-known that nontrivial complemented subspace exist in every dual Banach space of density greater than continuum, which is a result originally proved by Heinrich and Mankiewicz \cite[Corollary 3.8]{HM82}. Later, inspired Lindenstrauss' finite dimensional lemma, the proof was simplified by Sims and Yost \cite{Sims89} using the notion of a \emph{locally complemented subspace}. In \cite{God93} the authors realized there is a connection to the notion of $M$-ideals, and therefore, they used the term \emph{ideal} instead of locally complemented subspace. This notion was further developed and in \cite{Abr14} Abrahamsen, Lima and Nygaard introduced the notion of an \emph{almost isometric ideal} and found that almost isometric ideals inherit diameter $2$ properties and the Daugavet property.

Given a Banach space $X$ and its subspace $Y\subset X$, we say $Y$ is an \emph{almost isometric ideal} in $X$ if for any finite-dimensional subspace $E\subset X$ and any $\varepsilon>0$ there exists a bounded linear operator $T:E\to Y$ such
that $Tx=x$ for $x\in E\cap Y$ and $T$ is $(1+\varepsilon)$-isomorphic embedding (that is, we have $(1+\varepsilon)^{-1}\|x\|\leq \|Tx\|\leq(1+\varepsilon)\|x\|$ for every $x\in E$). If instead of $(1+\varepsilon)$-isomorphic embedding we require only $\|T\|\leq (1+\varepsilon)$ then we say $Y$ is \emph{ideal} in $X$. The connection to the structure of complemented subspaces is that $Y\subset X$ is an ideal if and only if $Y^\perp$ is kernel of a contractive projection on $X^*$, see e.g. \cite[Introduction]{God93}.

An essential result is \cite[Theorem 1.5]{Abr15} by which for every Banach space $X$ and its subspace $Z\subset X$ there exists almost isometric ideal $Y\subset X$ with $Z\subset Y$ and $\dens Y = \dens Z$. This has further applications, see e.g. \cite[Section 3]{Abr15} or also some more recent ones \cite{Band, Zoca21, Zoca22}. Our first main contribution is that we construct quite a large family of almost isometric ideals in every Banach space. There are more ways of formulating our result, the one which seems to be the most attractive one is through the notion of exceedingly rich families, which we introduce here.

\begin{defin}
Given a (nonseparable) topological space $X$ we say that a family of its closed subspaces $\SSS$ is \emph{exceedingly rich family} provided that:
\begin{enumerate}[label=\textnormal{(R-\alph*)}]
    \item\label{it:cofinalFirst} for every $Y\subset X$ there is $S\in\SSS$ such that $Y\subset S$ and $\max\{\dens Y,\omega\} = \dens S$;
    \item\label{it:upDirectedRich} if $\A\subset \SSS$ is up-directed, then $\overline{\bigcup \A} \in \SSS$.
\end{enumerate}
\noindent We say that a family $\SSS\subset \PP(X\times Y)$ is \emph{rectangular}, if for every $S\in\SSS$ there are $A\subset X$ and $B\subset Y$ such that $S = A\times B$.
\end{defin}

\noindent Our first main result is then the following.

\begin{thmx}\label{thm:main1}Let $X$ be a (nonseparable) Banach space. Then there exists a rectangular exceedingly rich family $\SSS$ of closed subspaces of $X\times X^*$ such for every $V\times W\in \SSS$ we have that $V$ is almost isometric ideal in $X$ and there is a contractive projection $P:X^*\to X^*$ with $\ker P = V^\perp$ and $\rng P \supset W$.
\end{thmx}

\noindent We note that a family satisfying condition \ref{it:cofinalFirst} was more-or-less found in \cite{Abr15} (see Theorem 1.5 and Remark 2.3 therein). The additional feature is condition \ref{it:upDirectedRich} which enables us to combine several results of this kind together due to the following simple observation.

\begin{lemma}\label{lem:combineRich}Let $X$ be (nonseparable) topological space and let $(\SSS_i)_{i\in\Nat}$ be a countable family of exceedingly rich families. Then $\bigcap_{i\in\Nat} \SSS_i$ is exceedingly rich family as well.
\end{lemma}
\begin{proof}It is obvious that condition \ref{it:upDirectedRich} holds for $\bigcap_{n} \SSS_n$. Let us prove \ref{it:cofinalFirst}. Pick $Z\subset X$. We inductively using find spaces $\{Y_i^n\colon i\leq n\}$ with $\dens Y_i^n\leq \dens Z$ such that $Y_i^n\in\SSS_i$ for every $i\leq n$ and
\[
Z\subset\quad Y_1^1 \subset \quad Y_2^1\subset Y_2^2\subset \quad \ldots \subset \quad Y_1^n\subset Y_2^n\subset \quad \ldots \subset Y_n^n\subset \quad Y_{1}^{n+1}\subset \ldots
\]
Then $Y_i^\infty : =\overline{\bigcup_{n=1}^\infty Y_i^n}\in \SSS_i$ and we have $Y_i^\infty = Y_j^\infty$ for every $i\neq j$. Thus, $Y_1^\infty\in \bigcap_{i\in\Nat} \SSS_i$, $Z\subset Y_1^\infty$ and $\dens Y_1^\infty = \dens Z$.
\end{proof}

\noindent The notion of exceedingly rich family is inspired by a related notion of a \emph{rich family} which is an analogue for the class of separable spaces, see \cite{C18} and references therein. We also note that existence of a rich family satisfying Theorem~\ref{thm:main1} follows from \cite[Theorem 2.4]{Band}, but when compared with our proof of Theorem~\ref{thm:main1}, the proof of \cite[Theorem 2.4]{Band} is more involved and less direct (and the result covers only separable spaces).

Very recently, Quilis and Zoca \cite{QuilisZoca} introduced a metric analogue of almost isometric ideals. We say that a subset $Y$ of a metric space $X$ is \emph{almost isometric local retract} in $X$, if for every finite set $E\subset X$ and every $\varepsilon>0$ there exists a Lipschitz mapping $T:E\to N$ such that $Tx=x$ for $x\in E\cap Y$ and $T$ is $(1+\varepsilon)$-biLipschitz embedding (that is, we have $(1+\varepsilon)^{-1}d(x,y)\leq d(Tx,Ty)\leq (1+\varepsilon)d(x,y)$ for $x,y\in E$). Using our approach, we quite easily obtain the following improvement of \cite[Theorem 5.5]{QuilisZoca}. This might be considered as the second main result of this paper.

\begin{thmx}\label{thm:main2}
Let $X$ be a metric space. Then there exists an exceedingly rich family $\SSS$ such that each $F\in\SSS$ is almost isometric local retract in $X$.
\end{thmx}

An important feature of our proofs, which might be of use in many similar situations is, that we do not actually construct exceedingly rich family, but rather family large in the sense of Skolem-like functions.

\begin{defin}\label{def:SkolemLikeFucntion}Given an infinite set $I$ we say that a function $\phi:\PP(I) \to \PP(I)$ is \emph{Skolem-like} provided that:
\begin{enumerate}[label=\textnormal{(S-\alph*)}]
    \item\label{it:skolemNotIncrease} $A\subset \phi(A)$ and $|A| = |\phi(A)|$ for every $A\subset I$;
    \item\label{it:skolemMonotone} $\phi$ is \emph{monotone}, that is, if $A\subset B$ then $\phi(A) \subset \phi(B)$ for $A,B\in\PP(I)$;
    \item\label{it:skolemUpDirected} given an up-directed family $\A \subset \PP(I)$, we have $\phi(\bigcup \A) = \bigcup_{A\in \A} \phi(A)$.
\end{enumerate}
Given a (nonseparable) topological space $X$ we say that a family of its closed subspaces $\SSS$ is \emph{large in the sense of Skolem-like functions} if there exists a Skolem-like function $\phi:\PP(X) \to \PP(X)$ such that
\[
\SSS = \{\overline{\phi(C)}\colon C\subset X\}.
\]
\end{defin}

The main ingredient of our proofs is the following which enables us to pass from one notion to another. Experience shows that it is easier to construct families large in the sense Skolem-like functions, so this seems to be quite a helpul tool and might be considered as the third main result of this paper.

\begin{thmx}\label{thm:main3}
Let $X$ be a topological space and let $\SSS$ be a family of its subspaces.
\begin{itemize}
    \item If $\SSS$ is exceedingly rich, then there exists $\SSS_0\subset \SSS$ large in the sense of Skolem-like functions.
    \item If $X$ is metrizable and $\SSS$ is large in the sense of Skolem-like functions, then there exists $\SSS_0\subset \SSS$ which is exceedingly rich.
\end{itemize}
\end{thmx}

\noindent Moreover, we also show that in metrizable topological spaces, the two methods mentioned above are equivalent also to the method of suitable models, we refer the reader to Section~\ref{sec:formulations} for more details. Our approach also gives some improvement of the main result from \cite{C18}, where a similar result was obtained only in the setting of spaces homeomorphic to Banach spaces and their separable subspaces, see Remark~\ref{rem:comparisionSeparable} for more details.

Finally, in the last part of this paper we show some applications of Theorem~\ref{thm:main1} and Theorem~\ref{thm:main2}, which are related to separable reduction (that is, a method of a proof where we extend validity of a result known to hold for separable spaces to the nonseparable setting not knowing the proof of the result in the separable case, we refer reader to \cite{C18} for some more details). A sample of those applications are the following two results.

\begin{thmx}\label{thm:main4}Let $X$ be a non-separable Banach space. Then there exists an exceedingly rich family $\SSS$ of subspaces of $X$ such that, for every $Y\in \SSS$ the following holds:
\[\begin{split}
X \text{ is a Gurari\u{\i} space}  & \Leftrightarrow Y \text{ is a Gurari\u{\i} space},\\
X \text{ is an }L_1\text{-predual space}  & \Leftrightarrow Y \text{ is an } L_1\text{-predual space},\\
X \text{ has Daugavet property }  & \Leftrightarrow Y \text{ has Daugavet property}.
\end{split}\]
(We emphasize that one family $\SSS$ is responsible for all the three equivalences.)
\end{thmx}

\begin{thmx}\label{thm:main5}Let $X$ be a non-separable metric space. Then there exists an exceedingly rich family $\SSS$ of subspaces of $X$ such that, for every $Y\in \SSS$ the following holds:
\[
X \text{ is an absolute ai-local retract} \Leftrightarrow Y\text{ absolute an ai-local retract}.
\]
\end{thmx}
\noindent Many other properties could be easily added to the lists above using our results and the above is just an incomplete sample witnessing applicability of our results, for more details we refer the interested reader to Remark~\ref{rem:moreApplications}.

The structure of the paper is the following. In Section~\ref{sec:prelim} we collect our notation and some preliminaries concerning Lipschitz-free spaces, which is a tool used in our proofs. Section~\ref{sec:formulations} is devoted mainly to the proof Theorem~\ref{thm:main3} comparing the methods mentioned in Theorem~\ref{thm:main3} moreover with the method involving suitable models (proofs in Section~\ref{sec:formulations} are using the set-theoretical method of suitable models, this method is not used anywhere else in the paper). The main outcome of Section~\ref{sec:aiIdeals} is the proof of Theorem~\ref{thm:main1}, the main outcome of Section~\ref{sec:aiLocRetr} is the proof of Theorem~\ref{thm:main2} and in Section~\ref{sec:appl} we find some applications of our results, in particular implying Theorem~\ref{thm:main4} and Theorem~\ref{thm:main5}.

\section{Preliminaries}\label{sec:prelim}

In this section we set up some notation that will be used throughout the paper. We also recall results concerning Lipschitz-free spaces which will be used further in the text.

Given a set $X$ and cardinal $\kappa$, we denote by $[X]^{\leq \kappa}$ the family of subsets of $X$ of cardinality at most $\kappa$ and by $[X]^{< \kappa}$ the family of subsets of $X$ of cardinality strictly less than $\kappa$. Cardinality of a set $X$ is denoted by $|X|$, family of its subset is denoted by $\PP(X)$. Given a Banach space $X$, $A\subset X$ and $B\subset X^*$ we put $A^\perp:=\{x^*\in X^*\colon x^*|_A\equiv 0\}$ and $B_\perp:=\{x\in X\colon x(b)=0\text{ for every }b\in B\}$. Further, we say $B\subset X^*$ is \emph{norming} if $X\setminus\{0\} \ni x\mapsto \sup_{x^*\in B\setminus\{0\}} \tfrac{|x^*(x)|}{\|x^*\|}$ defines an equivalent norm on $X$. We shall use also the following.

\begin{defin}Let $X$ and $Y$ be topological spaces. We say that a Skolem-like function $\phi:\PP(X\times Y)\to \PP(X\times Y)$ is \emph{rectangular} if for every $A\subset X\times Y$ there are $B\subset X$ and $C\subset Y$ such that $\phi(A) = B\times C$. In this case we denote by $\phi_X:\PP(X\times Y)\to \PP(X)$ and $\phi_Y:\PP(X\times Y)\to \PP(Y)$ the mappings satisfying that $\phi(A) = \phi_X(A)\times \phi_Y(A)$ for every $A\subset X\times Y$.

Note that of course if $\phi:\PP(X\times Y)\to \PP(X\times Y)$ is rectangular Skolem-like function then the corresponding large family $\SSS:=\{\overline{\phi(A)}\colon A\subset X\times Y\}$ is rectangular.
\end{defin}

Finally, let us give some preliminaries concerning Lipschitz-free spaces. Given a pointed metric space $(M,d,0)$, there exists a unique (up to isometry) Banach space $\F(M)$ (called the \emph{Lipschitz-free space over $M$}) such that there is an isometry $\delta:M\to \F(M)$ satisfying that $\delta(M\setminus\{0\})$ is linearly independent set with $\closedSpan \delta(M) = \F(M)$ and for every Banach space $X$ and a Lipschitz map $f:M\to X$ with $f(0)=0$ there exists $\widehat{f}:\F(M)\to X$ with $\widehat{f}\circ \delta = f$ and $\|\widehat{f}\| = \Lip(f)$, where $\Lip(f) = \sup\{\frac{\|f(x)-f(y)\|}{d(x,y)}\colon x\neq y\in M\}$. We note that in particular we have $\F(M)^* = \Lip_0(M)$, where $\Lip_0(M)$ is the vector space $\{f:M\to \Rea\colon f\text{ is Lipschitz and }f(0)=0\}$ endowed with the norm given by $\|f\|:=\Lip(f)$. We refer the interested reader e.g. to \cite{CDW} to some more details concerning the construction of Lipschitz-free spaces and their basic properties.

\section{Exceedingly rich families and equivalent methods}\label{sec:formulations}

The main purpose of this section is to compare three notions of largeness of a family of subspaces $\SSS\subset\PP(X)$ and prove those notions are in a sense equivalent. We warmly recommend the interested reader to the Introduction in \cite{C18}, where the motivation for such a result is further explained. The main outcome here is the proof of Theorem~\ref{thm:main3} and the proof of Lemma~\ref{lem:skeleton}, which are the only results needed in further sections.

Even thought it is possible to give a direct proof of Theorem~\ref{thm:main3}, the way how we found it was through its connection to the method of suitable models, so we prefer here to provide the reader with this insight as well. We refer the interested reader to \cite[Section 2]{C18}, where a proper introduction to the method of suitable models is given. Here, we just collect all the necessary notions needed to formulate and prove our result.

 Let $N$ be a fixed set and $\varphi$ a formula in the basic language of the set theory. By the relativization of $\varphi$ to $N$ we understand the formula $\varphi^N$ which is
obtained from $\varphi$ by replacing each chain of the form ``$\forall x$'' by ``$\forall x\in N$'' and each chain of the form ``$\exists x$'' by ``$\exists x\in N$''. Let $\varphi(x_1,\ldots,x_n)$ be a formula with all free variables shown, that is, a formula
whose free variables are exactly $x_1,\ldots,x_n$. We say $\varphi$ is absolute for $N$ if
\[
\forall a_1,\ldots,a_n\in N: \big(\varphi(a_1,\ldots,a_n) \Leftrightarrow \varphi^N(a_1,\ldots,a_n)\big).
\]
\begin{defin}
Let $\Phi$ be a finite list of formulas and $S$ be any countable set. Let $M\supset S$ be a set such that each $\varphi$ from $\Phi$ is absolute for $M$. Then we say that $M$ is a \emph{suitable model for $\Phi$ containing $S$} and we write $M \prec (\Phi; S)$.\\
(\textbf{Warning}: in \cite{C18} only countable models were considered, while here we consider suitable models which are not necessarily countable.)
\end{defin}
The method of suitable models is based mainly on the well-known theorem (see \cite[Chapter IV,
Theorem 7.]{kunenBook}) that for any finite list $\Phi$ of formulas and any set $S$ there exists $M\prec (\Phi;S)$ with $|M|\leq \max\{\omega,|S|\}$. We note that in our paper we will use the method of suitable models only through results proved elsewhere - so we do not require the reader to be too much familiar with it.

\begin{defin}
Given a (nonseparable) topological space $X$ we say that a family of its closed subspaces $\SSS$ is \emph{exceedingly large in the sense of suitable models} provided that there exists a finite list of formulas $\Phi$ and a countable set $C$ such that
\[
\SSS = \{\overline{X\cap M}\colon M\prec(\Phi;C)\}.
\]
Given $M\prec (\Phi;C)$ we put $X_M:=\overline{X\cap M}$.
\end{defin}

The main outcome of this Section is the following, which easily implies Theorem~\ref{thm:main3}.

\begin{thm}\label{thm:mainComparision}Let $X$ be a topological space and let $\SSS\subset \PP(X)$.
\begin{enumerate}
    \item\label{it:richImpliesMonotone} If $\SSS$ is exceedingly rich family, then there exists a family $\SSS_0\subset \SSS$ large in the sense of Skolem-like functions.
    \item\label{it:SkolemImpliesModels} If $\SSS$ is large in the sense of Skolem-like functions, then there exists a family $\SSS_0\subset \SSS$ exceedingly large in the sense of suitable models.
    \item\label{it:modelsImplyRich} If $X$ is metrizable and $\SSS$ is exceedingly large in the sense of suitable models, then there exists an exceedingly rich family $\SSS_0\subset \SSS$.
\end{enumerate}
\end{thm}

\noindent Before giving the proof of Theorem~\ref{thm:mainComparision} (see Subsection~\ref{subsec:prfComparision}), let us mention one of its consequences which we shall use later. It gives a connection of rectangular families with the notion of a projectional skeleton, see Lemma~\ref{lem:skeleton}. This notion was found in \cite{KubisSkeleton} and further studied by many authors in many papers, see e.g. \cite{CCS, Kal20}, where we refer the interested reader for more details. We start with an easy observation.

\begin{fact}\label{fact:productModels}Let $X$ and $Y$ be topological spaces. Then there is a finite list of formulas $\Phi$ and a countable set $S$ such that whenever $M\prec (\phi;S)$, then $(X\times Y)_M = X_M\times Y_M$.
\end{fact}
\begin{proof}By \cite[Lemma 7]{CCS0}, there is a finite list of formulas $\Phi$ and a countable set $S$ such that whenever $M\prec (\phi;S)$, then $(x,y)\in (X\times Y)\cap M$ if and only if $(x,y)\in (X\cap M)\times (Y\cap M)$. From this the fact easily follows.
\end{proof}

\begin{lemma}\label{lem:skeleton}Let $X$ be a Banach space and $D\subset X^*$ a norming closed subspace. Then the following conditions are equivalent.
\begin{enumerate}[label=\textnormal{(\roman*)}]
    \item\label{it:skeletonInduced} $D$ is subset of a set induced by a projectional skeleton in $X$.
    \item\label{it:richInduced} There exists a rectangular exceedingly rich family $\SSS\subset \PP(X\times D)$ such that for every $V\times W$ we have $X = \overline{V + W_\perp}$.
    \item\label{it:richInducedBest} There exists a rectangular exceedingly rich family $\SSS\subset \PP(X\times D)$ such that for every $V\times W$ there exists a projection $P:X\to X$ satisfying $\rng P = V$ and $\ker P = W_\perp$.
\end{enumerate}
\end{lemma}
\begin{proof}First, assume \ref{it:skeletonInduced} holds. By \cite[Lemma 15 and Proposition 16]{CCS}, there is a finite list of formulas $\Phi$ and a countable set $S$ such that whenever $M\prec (\phi;S)$, then $X = X_M + (D\cap M)_\perp$. Using Fact~\ref{fact:productModels}, we obtain a rectangular exceedingly large  family $\SSS\subset\PP(X\times D)$ in the sense of suitable models such that whenever $V\times W\in \SSS$, then $V + W_\perp = X$. Using Theorem~\ref{thm:mainComparision}, we obtain \ref{it:richInduced}.

Assume \ref{it:richInduced} holds. By Theorem~\ref{thm:mainComparision} together with Fact~\ref{fact:productModels} and \cite[Lemma 15]{CCS}, there is a finite list of formulas $\Phi$ and a countable set $S$ such that whenever $M\prec (\phi;S)$, then $(X\times D)_M = X_M\times D_M$ and there exists a projection $P:X\to X$ with $\rng P = X_M$ and $\ker P = (D\cap M)_\perp$. Thus, there is a a rectangular exceedingly large  family $\SSS\subset\PP(X\times D)$ in the sense of suitable models such that for every $V\times W\in \SSS$ there exists a projection $P:X\to X$ satisfying $\rng P = V$ and $\ker P = W_\perp$. By Theorem~\ref{thm:mainComparision}, this implies \ref{it:richInducedBest}.

Finally, assuming \ref{it:richInducedBest} holds, we easily check that for $\Gamma = \{(V,W)\in \SSS\colon V\times W\text{ is separable}\}$ ordered by inclusion, the projections $(P_{V\times W})_{V\times W\in\Gamma}$ satisfying $\rng P_{V\times W} = V$ and $\ker P_{V\times W} = W_\perp$ form a projectional skeleton (see e.g. \cite[Definition 1]{CCS}) and the set induced by the skeleton is in this case equal to $\bigcup_{V\times W\in \Gamma} (P_{V\times W})^*[X^*] = \bigcup_{V\times W\in \Gamma} (\ker P_{V\times W})^\perp
 = \bigcup_{V\times W\in \Gamma} \overline{W}^{w^*}\supset D$. Thus, \ref{it:skeletonInduced} holds.
\end{proof}

\subsection{Proof of Theorem~\ref{thm:mainComparision}}\label{subsec:prfComparision} Our proof of Theorem~\ref{thm:mainComparision} \eqref{it:richImpliesMonotone} is similar to the proof of \cite[Theorem 19]{C18}.

\begin{proof}[Proof of Theorem~\ref{thm:mainComparision} \eqref{it:richImpliesMonotone}]For every $x\in X$, let us pick separable $F_x \in\SSS$ with $x\in F_x$ and for every separable $F\in\SSS$, let us pick a
countable set $D_F \subset F$ which is dense in $F$. Moreover, for every separable $G, H\in\SSS$, pick separable $F_{G,H}\in\SSS$ with
$F_{G,H}\supset G \cup H$.  Pick an arbitrary infinite countable $Z\subset X$. Now, let us define a function $\psi:\PP(X)\to\PP(\SSS)$ by putting for every $C\subset X$
\[\begin{split}
\psi_0(C) & :=\{F_x\colon x\in C\cup Z\},\\
\psi_{k+1}(C)&:=\psi_k(C)\cup \{F_{G,H}\colon G,H\in\SSS\cap \psi_k(C),\;\text{ $G$ and $H$ are separable}\},\\
\psi(C)& :=\bigcup_{k=0}^\infty \psi_k(C).
\end{split}
\]
Finally, we define $\phi:\PP(X)\to \PP(X)$ by $\phi(C):=Z\cup C\cup \bigcup\{D_F\colon F\in \psi(C)\text{ is separable}\}$.

Let us prove that $\phi$ is indeed a Skolem-like function. Since $\psi$ is monotone, we see that $\phi$ is monotone as well. Moreover, we clearly have $C\subset \phi(C)$ and $\max\{|C|,\omega\} = |\phi(C)|$ for every $C\subset X$. Finally, pick an up-directed family $\A\subset \PP(X)$. By monotonicity we obviously have $\phi(\bigcup\A)\supset \bigcup_{A\in\A}\phi(A)$. For the other inclusion, using that $\A$ is up-directed we first easily by induction prove that $\psi_k(\bigcup \A) \subset \bigcup_{A\in \A}\psi(A)$ for any $k\in\Nat\cup\{0\}$. Thus, given $x\in \phi(\bigcup \A)$ either we have $x\in Z\cup \bigcup\A$ (and therefore we indeed have $x\in \bigcup_{A\in\A}\phi(A)$) or we pick $k\in\Nat\cup\{0\}$ such that $x\in D_F$ for some separable $F\in \psi_{k+1}(\bigcup \A) \subset \bigcup_{A\in\A}\psi(A)$ in which case we have $x\in \phi(A)$ for some $A\in\A$. This completes the proof that $\phi$ is indeed a Skolem-like function.

It remain to prove that $\SSS_0:=\{\overline{\phi(C)}\colon C\subset X\}$ satisfies that $\SSS_0\subset \SSS$. Pick $C\subset X$. First, we \emph{claim} that
\begin{equation}\label{eq:XMisGivenByRich}
\overline{\phi(C)} = \overline{\bigcup\{F\colon F\in\psi(C), F\text{ is separable}}\}.
\end{equation}
Indeed, it follows from the construction that given $x\in \phi(C)$ there is separable $F\in\psi(C)$ such that $x\in F$, which proves inclusion ``$\subset$'' in \eqref{eq:XMisGivenByRich}. Conversely, given separable $F\in\psi(C)$ we have $D_F\subset \phi(C)$ and so $F = \overline{D_F}\subset \overline{\phi(C)}$, which proves inclusion ``$\supset$'' in \eqref{eq:XMisGivenByRich}. This proves the claim and so \eqref{eq:XMisGivenByRich} holds. Now, we observe that by the definition of $\psi$ we have that $\A:=\{F\colon F\in \psi(C)\text{ is separable}\}$ is up-directed subset of $\SSS$ and therefore
\[
\overline{\phi(C)} \stackrel{\eqref{eq:XMisGivenByRich}}{=} \overline{\bigcup \A}\in\SSS.
\]
Since $C\subset X$ was arbitrary, this shows that $\SSS_0\subset \SSS$.
\end{proof}

Proof of Theorem~\ref{thm:mainComparision} \eqref{it:SkolemImpliesModels} is similar to the proof of \cite[Proposition 3.1]{CK04}.

\begin{proof}[Proof of Theorem~\ref{thm:mainComparision} \eqref{it:SkolemImpliesModels}]Let $\tau$ be the topology of $X$ and $\phi:\PP(X)\to\PP(X)$ be a Skolem-like function such that $\SSS = \{\overline{\phi(A)}\colon A\subset X\}$. Let $\Phi$ be the union of finite lists of formulas from the statements of results used in the proof below (namely \cite[Lemma 7]{CCS0}). Let $S$ be the union of $\{X,\tau,\phi\}$ together with the countable sets from the statements of results used in the proof below (namely \cite[Lemma 7]{CCS0}). Let $M \prec (\Phi; S)$. In order to finish the proof it suffices to show that then $X_M = \overline{\phi(X\cap M)}$.

Note that by \cite[Lemma 7]{CCS0} we have: $A\in M$ if and only if $A\subset M$ whenever $A\subset X$ is finite; $A\in M$ implies that $A\subset M$ whenever $A\subset X$ is countable; $\phi(A)\in M$ for every $A\in M\cap \PP(X)$.

Now, we shall observe that
\begin{equation}\label{eq:xMequality}
X_M = \overline{\bigcup\{\phi(A)\colon A\subset X\cap M\text{ is finite set}\}}.
\end{equation}
Indeed, for the inclusion ``$\supset$'' we notice that given a finite set $A\subset X\cap M$ we have $A\in M$ and therefore $\phi(A)\in M$ which, since $\phi(A)$ is countable, implies $\phi(A)\subset X\cap M$. For the inclusion ``$\subset$'' we notice that given $x\in X\cap M$, we have $x\in \phi(A)$ where $A = \{x\}\subset X\cap M$. Thus, \eqref{eq:xMequality} holds.

Further, using \ref{it:skolemUpDirected} we also have
\[
\overline{\phi(X\cap M)} = \overline{\bigcup\{\phi(A)\colon A\subset X\cap M\text{ is finite set}\}}.
\]
This together with \eqref{eq:xMequality} shows that $X_M = \overline{\phi(X\cap M)}$, which is what we needed to finish the proof.
\end{proof}

Finally, we shall prove Theorem~\ref{thm:mainComparision} \eqref{it:modelsImplyRich}. The basic idea in the argument below is similar to the proof of \cite[Theorem 21]{C18}. The additional ingredient is the use of Lipschitz-free spaces, which enable us to work not only with Banach spaces, but rather with metric spaces in general. Also, we present the proof using suitable models, which is less technical approach when compared to the proof of \cite[Theorem 21]{C18}, where a direct inductive construction was used. Reader not familiar with the method of suitable models can then compare both methods (suitable models used in the proof below and inductive construction used in \cite[Theorem 21]{C18}) and convince himself that when using suitable models we really just hide the technicalities concerning a precise inductive construction in an abstract result, otherwise both methods use the same ideas (in the proof below the inductive construction is hidden in the proof of \cite[Lemma 13]{CCS} used therein).

\begin{proof}[Proof of Theorem~\ref{thm:mainComparision} \eqref{it:modelsImplyRich}]Let $\Phi'$ be a finite list of formulas and $S'$ a countable set such that $\SSS = \{X_M\colon M\prec(\Phi';S')\}$.

Let $d$ be a metric on $X$ generating its topology. By \cite{Tor81}, all Banach spaces of the same density are topologically homeomorphic. Thus, there exists a homeomorphism $h:\F(X)\to \ell_2(\kappa)$, where $\kappa = \dens X$. Let $e:\kappa \to \ell_2(\kappa)$ be one-to-one mapping such that $e[\kappa]$ is the canonical orthonormal basis of $\ell_2(\kappa)$.

Let $\Phi$ be the union of $\Phi^\prime$ and the finite lists of formulas from the statements of results used in the proof below (namely \cite[Lemma 7 and 8]{CCS0} and \cite[Lemma 12]{CCS}). Let $S$ be the union of $S'\cup\{h,\kappa,e\}\cup\{X,d\}\cup\{\F(X),\delta,\|\cdot\|,+,\cdot\}\cup\{\ell_2(\kappa),\|\cdot\|,+,\cdot\}$ together with the countable sets from the statements of results used in the proof below (namely \cite[Lemma 7 and 8]{CCS0} and \cite[Lemma 12]{CCS}). Let $M \prec (\Phi; S)$.

Note that by \cite[Lemma 7 and 8]{CCS0} we have $\delta[X]\in M$, $e[\kappa]\in M$ and by \cite[Lemma 12 (2)(ii)]{CCS} we therefore obtain 
\[\begin{split}
\overline{\F(X)\cap M} = \closedSpan\{\delta(x)\colon x\in X\cap M\},\quad \overline{\ell_2(\kappa)\cap M}= \closedSpan \{e_i\colon i\in\kappa\cap M\}
\end{split}\]
and by \cite[Lemma 12 (1)]{CCS} we have $h[\overline{\F(X)\cap M}] = \overline{\ell_2(\kappa)\cap M}$. By
\cite[Theorem IV.7.4]{kunenBook}, there exists a set $R$ such that $R \prec (\Phi; S\cup \kappa)$. Let $\psi: \PP(R) \to \PP(R)$ be the Skolem
function given by \cite[Lemma 13]{CCS}, in particular we have
\begin{enumerate}
    \item\label{it:areModels} \[\forall A\subset R:\qquad \psi(A)\prec (\Phi; S),\; |\psi(A)| \leq \max\{|A|,\omega\},\]
    \item $\psi(A)\subset \psi(B)$ whenever $A\subset B\subset R$,
    \item\label{it:upDirectedModelsSkolem} whenever $\A\subset \PP(R)$ is such that $\{\psi(A)\colon A\in\A\}$ is up-directed, then $\psi(\bigcup \A) = \bigcup_{A\in A} \psi(A)$,
    \item\label{it:subsetOnOrdinals} for every $A,B\subset \kappa$ we have $\psi(A\cup S)\subset \psi(B\cup S)$ if and only if $\psi(A\cup S)\cap \kappa\subset \psi(B\cup S)\cap \kappa$.
\end{enumerate}
Consider the following family of subsets of $R$
\begin{equation}\label{eq:defM}
\M := \{\psi(A\cup S)\colon A\subset \kappa\}.
\end{equation}
and put $\SSS_0:=\{X_M\colon M\in\M\}$. By \eqref{it:areModels} we have $\SSS_0\subset\SSS$ and so it remains to prove that $\SSS_0$ is exceedingly rich family.

Pick infinite $C\subset X$ and let $D\subset C$ be its dense subset with $|D| = \dens C$. We want to find $M\in\M$ such that $D\subset X_M$ and $\dens X_M\leq |D|$. Pick $A\subset \kappa$ satisfying $h[\delta[D]]\subset \closedSpan\{e_i\colon i\in A\}$ and $|A| = |D|$. We \emph{claim} that for $M:=\psi(S\cup A)$ we have $D\subset X_M$. Indeed, first we notice that
\[
\delta(D)\subset h^{-1}[\closedSpan\{e_i\colon M\cap\kappa\}] = \closedSpan\{\delta(x)\colon x\in X_M\}.
\]
This already implies $D\subset X_M$ as otherwise for $s\in D\setminus X_M$ we have $d(s,X_M)>0$ so we may find a Lipschitz function $f$ with $f|_{X_M}\equiv 0$ and $f(s)=1$, but then $f$ separates $\closedSpan\{\delta(x)\colon x\in X\cap M\}$ from $\delta(s)$, a contradiction. Thus, we indeed have $D\subset X_M\in\SSS_0$ and $\dens X_M\leq |\psi(S\cup A)| \leq |C|$. 

Finally, let $\A\subset \SSS_0$ be up-directed. There is $I\subset \M$ with $\A = \{X_M\colon M\in I\}$. Notice that for $M,N\in \M$ we have $M\subset N$ if and only if $X_M\subset X_N$. Indeed, one implication is trivial and for the other one we note that $X_M\subset X_N$ implies that
\[\begin{split}
\closedSpan\{e_i\colon i\in \kappa\cap M\} & = h[\closedSpan\{\delta(x)\colon x\in X_M\}]\\
& \subset h[\closedSpan\{\delta(x)\colon x\in X_N\}] = \closedSpan\{e_i\colon i\in \kappa\cap M\},
\end{split}\]
which in turn implies $\kappa\cap M\subset \kappa\cap N$ and by \eqref{it:subsetOnOrdinals} we therefore obtain $M\subset N$. Thus, since the mapping $I\ni M\mapsto X_M\in \A$ is order-isomorphism, $I$ is up-directed. For every $M\in I$ there is $A_M$ with $M = \psi(S\cup A_M)$ and we obtain
\[
\bigcup_{M\in I} M = \bigcup_{M\in I}\psi(S\cup A_M) \stackrel{\eqref{it:upDirectedModelsSkolem}}{=} \psi(S\cup \bigcup_{M\in I} A_M) =: M_\infty\in \M,
\]
which implies that
\[
\overline{\bigcup \A} = \overline{\bigcup_{M\in I} \overline{X\cap M}} = \overline{X\cap (\bigcup_{M\in I} M)} = \overline{X\cap M_{\infty}}\in \SSS_0.
\]
This finishes the proof that $\SSS_0\subset \SSS$ is exceedingly rich family.
\end{proof}

\begin{remark}\label{rem:comparisionSeparable}
We note that minor modifications in the proof of Theorem~\ref{thm:mainComparision} \eqref{it:modelsImplyRich} lead to the proof of the following, which improves \cite[Theorem 15]{C18} from spaces homeomorphic to a Banach space to metrizable topological spaces: supposing $X$ is a metrizable topological space and $\SSS\subset \PP(X)$ is large in the sense of suitable models, there exists a rich family $\SSS_0\subset \SSS$. The only modification which is needed is: in the definition of $\M$ in \eqref{eq:defM} one needs to write $A\in [\kappa]^{\leq \omega}$ instead of $A\subset \kappa$.

We refer the interested reader to \cite{C18}, where the needed notions are defined.
\end{remark}

\section{Almost Isometric Ideals}\label{sec:aiIdeals}

The main purpose of this section is to prove Theorem~\ref{Thm:ExeedRichFamilyai}, Theorem~\ref{thm:richAiIdeals} and its consequence Corollary~\ref{cor:projectionAiIdeal}, from which Theorem~\ref{thm:main1} easily follows (using Lemma~\ref{lem:combineRich}). Basic strategy of our proof is the same as \cite[proof of Theorem 1.5]{Abr15}. In order not to get lost in the inductive construction which we shall use, we start by extracting the main technicalities in Lemma~\ref{lem:calculations}.

\begin{lemma}\label{lem:calculations}
Let $X$ be a Banach space, $U,V\subset X$ finite-dimensional subspaces with $U\cap V = \{0\}$. Then for any $m\in\Nat$ there is a function $\zeta:[0,1)\to[0,1)$ continuous at zero with $\zeta(0)=0$ such that whenever $\varepsilon\in (0,1]$,
\begin{itemize}
    \item $n\in\Nat$ is such that $n\geq \tfrac{1}{\varepsilon}$,
    \item $N\subset V$ is an $\varepsilon$-net of $nB_V$,
    \item $A\subset \ell_1^{\dim U}$ is an $\varepsilon$-net of $S_{\ell_1^{\dim U}}$
    \item $\{e_1,\ldots,e_{\dim U}\}\subset mB_U$ is a basis of $U$,
    \item $\{f_1,\ldots,f_{\dim U}\}\subset mB_X$ are such that for every $x\in N$ and $a\in A$ we have
        \[
            \Big|\|x+\sum_{i=1}^{\dim U}a_i e_i\| - \|x+\sum_{i=1}^{\dim U}a_i f_i\|\Big| \leq \varepsilon,
        \]
\end{itemize}
then the mapping $T:U\oplus V\to X$ defined by 
\[
T(x+\sum_{i=1}^{\dim U}a_i e_i):=x+\sum_{i=1}^{\dim U}a_i f_i,\quad x\in V,\; a\in \Rea^{\dim U}
\]
is $(1+\zeta(\varepsilon))$-isomorphism.

Moreover, if $x^*\in X^*$ is such that $|x^*(e_i-f_i)|<\varepsilon$, then $\|(x^*\circ T - x^*)|_{U\oplus V}\| < \zeta(\varepsilon)\|x^*\|$.
\end{lemma}
\begin{proof}Let us denote by $P_U$ the projection $P_U:U\oplus V\to U$. Further, since all the norms on a finite-dimensional space are equivalent, pick $C>0$ such that $\|\sum_{i=1}^{\dim U} a_ie_i\|\geq C\|a\|_1$ for any $a\in\Rea^{\dim U}$.

Pick any $a\in S_{\ell_1^{\dim U}}$ and $x\in nB_V$. Let $b\in A$ and $y\in N$ be such that $\|a-b\|_1\leq \varepsilon$ and $\|x-y\|\leq \varepsilon$. Then we have
	\begin{equation} \label{key_lemma_aux_calc_3}
		\begin{split} 
			\Bigg| \Big\| x & + \sum_{i=1}^{\dim U} a_i e_i \Big\| - \Big\| x + \sum_{i=1}^{\dim U} a_i f_i \Big\| \Bigg| \\
			& \leq \| x - y \| + \Big\| \sum_{i=1}^{\dim U} a_i e_i - \sum_{i=1}^{\dim U} b_i e_i \Big\| + \|x-y\|  + \Big\| \sum_{i=1}^{\dim U} a_i f_i - \sum_{i=1}^{\dim U} b_i f_i \Big\| \\
			& \quad + \Bigg| \Big\| y + \sum_{i=1}^{\dim U} b_i e_i \Big\| - \Big\| y + \sum_{i=1}^{\dim U} b_i f_i \Big\| \Bigg| \\
			& < \varepsilon + \varepsilon m + \varepsilon + \varepsilon m + \varepsilon = \varepsilon(3 + 2m).
		\end{split}
	\end{equation}
Moreover, for any $z\in V$ and $c\in \Rea^{\dim U}$ we have
\[
    \Big\|z + \sum_{i=1}^{\dim U}c_i e_i\Big\| \geq \frac{1}{\|P_U\|}\Big\|\sum_{i=1}^{\dim U}c_i e_i\Big\|\geq \frac{C\|c\|_1}{\|P_U\|},
\]
so from \eqref{key_lemma_aux_calc_3} we obtain
\[
    \Bigg| \Big\| x + \sum_{i=1}^{\dim U} a_i e_i \Big\| - \Big\| x + \sum_{i=1}^{\dim U} a_i f_i \Big\| \Bigg|\leq \frac{\varepsilon(3+2m)\|P_U\|}{C}\Big\| x + \sum_{i=1}^{\dim U} a_i e_i \Big\|,
\]
which implies that for $f(\varepsilon):=\min\{\frac{\varepsilon(3+2m)\|P_U\|}{C},\tfrac{1}{2}\}$ we have
\begin{equation}
    (1-f(\varepsilon))\Big\| x + \sum_{i=1}^{\dim U} a_i e_i \Big\|\leq \Big\| x + \sum_{i=1}^{\dim U} a_i f_i \Big\|\leq (1+f(\varepsilon))\Big\| x + \sum_{i=1}^{\dim U} a_i e_i \Big\|
\end{equation}
and for $h(\varepsilon) = \tfrac{f(\varepsilon)}{1-f(\varepsilon)}$ we have $(1+h(\varepsilon))^{-1} = (1-f(\varepsilon))$ and therefore
\begin{equation}\label{key_lemma_final_ineq}
    \frac{1}{1+h(\varepsilon)}\Big\| x + \sum_{i=1}^{\dim U} a_i e_i \Big\|\leq \Big\| x + \sum_{i=1}^{\dim U} a_i f_i \Big\|\leq (1+h(\varepsilon))\Big\| x + \sum_{i=1}^{\dim U} a_i e_i \Big\|
\end{equation}
whenever $x\in nB_V$ and $a\in S_{\ell_1^{\dim U}}$.

    Pick arbitrary $a \in S_{\ell_1^{\dim U}}$ and $x \in X \setminus mB_U$. We observe
	\[
	\|x\| - m \leq \| x \| - \left\| \sum_{i=1}^{\dim U} a_i e_i \right\| \leq \left\| x + \sum_{i=1}^{\dim U} a_i e_i \right\| \leq \|x\| + m.
	\]
	The same argument yields
	\[
	\|x\| -m \leq \left\| x + \sum_{i=1}^{\dim U} a_i f_i \right\| \leq \|x\| + m.
	\]
	This means
	\begin{equation} \label{key_lemma_two_ineq}
		\frac{\left\| x + \sum_{i=1}^{\dim U} a_i f_i \right\|}{\left\| x + \sum_{i=1}^{\dim U} a_i e_i \right\|} \leq \frac{\|x\| + m}{\|x\| - m} \text{ and } \frac{\left\| x + \sum_{i=1}^{\dim U} a_i e_i \right\|}{\left\| x + \sum_{i=1}^{\dim U} a_i f_i \right\|} \leq \frac{\|x\| + m}{\|x\| - m}.
	\end{equation}
	We consider the real function $g: z \mapsto \frac{z+m}{z-m} - 1$. Because $\| x \| > n\geq \tfrac{1}{\varepsilon}$ and the function $g$ is decreasing to $0$, we have $\frac{\|x\| + m}{\|x\| - m} = g(\|x\|) +1<g(\tfrac{1}{\varepsilon}) + 1$. From this and \eqref{key_lemma_two_ineq} we obtain	\begin{equation}\label{eq:keyLemmaOdhadVetsiNezGJedna}
	\frac{1}{1 + g(\tfrac{1}{\varepsilon})} \Big\| x + \sum_{i=1}^{\dim U} a_i e_i \Big\|\leq \Big\| x + \sum_{i=1}^{\dim U} a_i f_i \Big\| \leq (1 + g(\tfrac{1}{\varepsilon})) \Big\| x + \sum_{i=1}^{\dim U} a_i e_i \Big\|.
	\end{equation}

 Thus, for $\zeta(\varepsilon):=2\max\{h(\varepsilon),g(\tfrac{1}{\varepsilon}),\frac{\varepsilon}{C} \|P_U\|\}$ from \eqref{key_lemma_final_ineq} and \eqref{eq:keyLemmaOdhadVetsiNezGJedna} we obtain that 
 \[
    \frac{1}{1 + \zeta(\varepsilon)} \Big\| x + \sum_{i=1}^{\dim U} a_i e_i \Big\| < \Big\| T\Big(x + \sum_{i=1}^{\dim U} a_i e_i\Big) \Big\| < (1 + \zeta(\varepsilon)) \Big\| x + \sum_{i=1}^{\dim U} a_i e_i \Big\|
 \]
 for every $a\in S_{\ell_1^{\dim U}}$ and every $x\in V$, which easily using a homogeneous argument implies that $T$ is $(1+\zeta(\varepsilon))$-isomorphism.

 In order to prove the ``Moreover'' part, pick $x^*\in X^*$ with $|x^*(e_i-f_i)|<\varepsilon$ and $y\in B_{U\oplus V}$. Let us denote $y = x + \sum_{i=1}^{\dim U} a_i e_i$. Then we have
	\begin{align*}
		\left| x^*(Ty) - x^*(y) \right| &= \left| x^*\left(x + \sum_{i=1}^{\dim U} a_i f_i\right) - x^*\left(x + \sum_{i=1}^{\dim U} a_i e_i\right) \right| \\
		&= \left| \sum_{i=1}^{\dim U} a_i x^*(e_i - f_i) \right| < \|x^*\|\sum_{i=1}^{\dim U} |a_i| \varepsilon = \|x^*\|\varepsilon \|a\|_{\ell_1^{\dim U}} \\
		&\leq \frac{\varepsilon\|x^*\|}{C} \left\| \sum_{i=1}^{\dim U} a_i e_i \right\| = \frac{\varepsilon\|x^*\|}{C} \| P_U(y) \| \leq \frac{\varepsilon\|x^*\|}{C} \|P_U\| \|y\|.
	\end{align*}
	Thus, we have that $\|(x^*\circ T - x^*)|_{U\oplus V}\| \leq \frac{\varepsilon}{C} \|P_U\|\|x^*\| < \zeta(\varepsilon)\|x^*\|$.
\end{proof}

The inductive construction we need is contained in the following analogy to \cite[Lemma 2.1]{Abr15}.

\begin{lemma}[Key Lemma]\label{lem:keyLemma}
	Let $X$ be a Banach space. Then there exists a rectangular Skolem-like function $\phi:\PP(X\times X^*)\to \PP(X\times X^*)$ such that for every $A\subset X\times X^*$ the sets $\phi_X(A)$ and $\phi_{X^*}(A)$ are $\Rat$-linear subspaces and the following holds :
 
    Let $\varepsilon > 0$, $D \in [\phi_{X^*}(A)]^{<\omega}$, $B \in [\phi_X(A)]^{<\omega}$ and $E\subset X$ be a finite-dimensional subspace with $E \supset B$. Then there is a continuous linear mapping $T: E \to \overline{\phi_X(A)}$ such that
	\begin{enumerate}[label=\textnormal{(K\alph*)}]
		\item\label{list:keyLemma1} $Tx = x$ for all $x\in B$,
		\item\label{list:keyLemma2} $(1+\varepsilon)^{-1}\|x\| \leq \|T x \| \leq (1+\varepsilon) \|x\|$  for all $x \in E$,
		\item\label{list:keyLemma3}$\|(d\circ T - d)|_E\|\leq \varepsilon \|d\|$ for all $d\in D$,
	\end{enumerate}
\end{lemma}
\begin{proof}We start the proof by defining certain sets which we will later use to construct the desired Skolem-like function. Given sets $A\subset X\times X^*$, finite sets $B\subset X$ and $D\subset X^*$, for each $n,k \in \Nat$ and $\delta\in \Rat_+$, we pick a finite $\delta$-net $A_{k,\delta}$ of $S_{\ell_1^k}$ and a finite $\delta$-net $N_{n,\delta}(B)\subset \Span B$ of $nB_{\Span B}$. Now since the mapping $\psi_{n,k,\delta}: (B_X)^{k} \to  (\mathbb{R}^{N_{n,\delta}(B) \times A_{k,\delta}} \times \mathbb{R}^{D \times k},\|\cdot\|_\infty)$, defined by the formula
\[
		\psi_{n,k,\delta}(u) := \left(\left(\left\|x + \sum_{i=1}^{k} a_iu_i \right\|\right)_{x\in N_{n,\delta}(B),a\in A_{k,\delta}},\left(d(u_i)\right)_{d\in D, i\leq k}\right),
\]
has separable range, we may pick a countable set $F_{n,k,\delta}(B,D)\subset B_X$ such that $\psi_{n,k,\delta}(F_{n,k,\delta}(B,D)^k)$ is dense in the range of $\psi_{n,k,\delta}$.
Next, if $B(A):=\Span_\Rat \pi_X(A)\subset X$ and $D(A):=\Span_\Rat \pi_{X^*}(A)\subset X^*$ for every $A\subset X\times X^*$, we inductively define $\xi(A)\subset X$ for $A\subset X\times X^*$ as follows:
\[\begin{split}
\xi_0(A)&:=B(A),\\
\xi_{n+1}(A)&:=\Span_\Rat\Big(\xi_n(A)\cup \bigcup \Big\{F_{n,k,\delta}(B,D)\colon n,k\in\Nat,\; \delta\in\Rat_+,\\
&\qquad\qquad B\subset \xi_n(A) \text{ and } D\subset D(A) \text{ are finite sets}\Big\}\Big),\\
\xi(A)&:=\bigcup_{n=0}^\infty \xi_n(A).
\end{split}\]
Finally, we put $\phi(A):=\xi(A)\times D(A)$ for $A\subset X\times X^*$. We claim that $\phi$ is indeed a rectangular Skolem-like function. It is evident that $\phi$ satisfies conditions \ref{it:skolemNotIncrease} and \ref{it:skolemMonotone} from Definition \ref{def:SkolemLikeFucntion}. To check \ref{it:skolemUpDirected} we pick an up-directed family $\mathcal{A}\subset P(X\times X^*)$. From \ref{it:skolemMonotone} we have $\phi(\bigcup \mathcal{A})\supset \bigcup_{A\in \mathcal{A}}\phi(A)$. To check the reverse inclusion, we start by proving the following claim: for every $n<\omega$, if $x\in \xi_n(\bigcup \mathcal{A})$, then there is $A\in \mathcal{A}$ such that $x\in \xi_n(A)$.

We proceed by induction on $n$. Since the map $A\mapsto B(A)$ satisfies \ref{it:skolemUpDirected}, and $\mathcal{A}$ is up-directed, the claim is easily seen to be true for $n=0$. Assume that it holds for $n$, and let $x$ be an element in $\xi_{n+1}(\bigcup \mathcal{A})$. Then
\[x=q_0y+\sum_{j=1}^mq_jz_j\]
for some $q_0,q_1,\ldots,q_m\in\Rat$, $y\in \xi_n(\bigcup \mathcal{A})$ and $z_j\in F_{n_j,k_j,\delta_j}(B_j,D_j)$, $j\leq m$ where $n_j,k_j\in\Nat$, 
$\delta_j\in\Rat_+$, and $B_j$ and $D_j$ are finite subsets of $\xi_n(\bigcup \mathcal{A})$ and $D(\bigcup \mathcal{A})$ respectively. By the induction hypothesis, there is $A_0\in \mathcal{A}$ such that $y\in \xi_n(A_0)$. Moreover, since $\mathcal{A}$ is up-directed, by applying the induction hypothesis again, we deduce that there is $A_1\in \mathcal{A}$ such that, $A_0\subset A_1$ and $\bigcup_{j=1}^mB_j\subset \xi_n(A_1)$. Since $A\mapsto D(A)$ satisfies \ref{it:skolemUpDirected}, we may also assume that $\bigcup_{j=1}^mD_j\subset D(A_1)$. Hence $x\in \xi_{n+1}(A_1)$ and the claim is established by induction.

Now let $(x,x^*)\in \phi(\bigcup \mathcal{A})$ be arbitrary. Then $x^*\in D(\bigcup \mathcal{A})$ and there is $n<\omega$ such that 
$x\in \xi_n(\mathcal{A})$. Since the map $A\mapsto D(A)$ satisfies \ref{it:skolemUpDirected}, there is $A_0\in\A$ such that $x^*\in D(A_0)$. In accordance with our claim, there exists $A_1\in \mathcal{A}$ such that $x\in \xi_n(A_1)$. Finally, since $\mathcal{A}$ is up-directed, there is $A_2\in\mathcal{A}$ such that $A_0\cup A_1\subset A_2$ and $(x,x^*)\in \xi_n(A_1)\times D(A_0)\subset \phi(A_2)$.
This establishes that $\phi(\bigcup \mathcal{A})\subset \bigcup_{A\in \mathcal{A}}\phi(A)$. 
The fact that $\phi$ is rectangular is evident from the definition of $\phi$.

Now, pick $\varepsilon>0$, $A\subset X\times X^*$, $D\subset D(A) = \phi_{X^*}(A)$ finite, $B\subset \phi_X(A) = \xi(A)$ finite and finite-dimensional space $E\subset X$ with $E\supset B$. Since $\Span B$ is finite-dimensional, there is $U\subset E$ such that $E = \Span B\oplus U$ and we may pick a basis $\{e_1,\ldots,e_k\}\subset B_U$ of the space $U$. Let $\zeta$ be the function from Lemma~\ref{lem:calculations} applied to $\Span B\oplus U$ and let $\delta>0$ be such that $\zeta(\delta)<\varepsilon$. Now, pick $n\in \Nat$ with $n\geq \tfrac{1}{\delta}$. By the choice of $F_{n,k,\delta}(B,D)$, there are $f_1,\ldots,f_k\in F_{n,k,\delta}(B,D)\subset \xi(A)$ such that $\|\psi_{n,k,\delta}((f_i)) - \psi_{n,k,\delta}((e_i))\|_\infty$. Then, by Lemma~\ref{lem:calculations}, since $\zeta(\delta)<\varepsilon$, we have that the mapping $T:E = \Span B\oplus U\to \overline{\xi(A)} = \overline{\phi_X(A)}$ given by
 \[
 T(x+\sum_{i=1}^{k}a_i e_i):=x+\sum_{i=1}^{k}a_i f_i,\quad x\in \Span B,\; a\in \Rea^{k}
 \]
 is $(1+\varepsilon)$-isomorphic embedding and so \ref{list:keyLemma2} holds. Finally, since for any $d\in D$ we have $|d(e_i)-d(f_i)|<\delta$, $i=1,\ldots,k$, the moreover part of Lemma~\ref{lem:calculations} implies that \ref{list:keyLemma3} holds as well.
\end{proof}

The first main result of this section is now contained in the following.

\begin{thm}\label{Thm:ExeedRichFamilyai}Let $X$ be a Banach space. Then there exists an exceedingly rich family $\SSS$ of subspaces of $X$ such that each $F\in \SSS$ is ai-ideal in $X$.
\end{thm}
\begin{proof}By Theorem~\ref{thm:main3}, it suffices to find a family $\SSS$ large in the sense of Skolem-like functions such that  each $F\in \SSS$ is ai-ideal in $X$.

Pick the Skolem-like function $\phi:\PP(X\times X^*)\to \PP(X\times X^*)$ from the Key Lemma~\ref{lem:keyLemma} and consider its projection onto the first coordinate given by $\psi(A):=\phi_X(A\times \{0\})$. It is easy to check that $\psi:\PP(X)\to\PP(X)$ is Skolem-like function as well. Thus, in order to finish the proof it suffices to check that $\overline{\psi(A)}$ is ai-ideal for every $A\subset X$.

Fix $A\subset X$ and put $Y:=\overline{\psi(A)}$. Given a finite dimensional subspace $E\subset X$ and $\epsilon>0$ (without loss of generality, we may assume that $\epsilon<1$), we let $\delta=\min\{\frac{\epsilon}{2(1+\dim(E))},\frac{\epsilon}{2}(\frac{1-\epsilon}{1+\epsilon})\}$ and pick a finite $\frac{\delta}{2(\delta+2)}$-net $N$ for $B_{E\cap Y}$. Since 
$\psi(A)$ is dense subset of $Y$, there exists a function $\varphi:N\to \psi(A)$ such that $\|\varphi(x)-x\|\leq \frac{\delta}{2(\delta+2)}$ for every $x \in N$. Since $\phi$ was the function from Key Lemma~\ref{lem:keyLemma}, for $E^\prime=\Span\{E\cup \varphi[N]\}$, $B=\varphi[N]$, $D=\emptyset$ and $\delta>0$ there exists a $(1+\delta)$-isomorphic embedding $T:E\to Y$ such that $Tx = x$ for $x\in B$.

Now, we shall check that then $\|(Id-T)|_{E\cap Y}\|\leq \delta$. Indeed, for $x\in B_{E\cap Y}$, there is $n\in N$ such that $\|x-n\|\leq\frac{\delta}{2(\delta+2)}$ and hence $\|x-\varphi(n)\|\leq \|x-n\|+\|\varphi(n)-n\|\leq \frac{\delta}{\delta+2}$. In other words, there exists $b \in B$ such that $\|x-b\|\leq  \frac{\delta}{\delta+2}$. Since 
$Tb=b$ we have
\[\|(\mathrm{Id}-T)x\|\leq \|x-b\|+\|Tb-Tx\|\leq (1+\|T\|)\|x-b\|\leq \delta,\]
so we really have $\|(Id-T)|_{E\cap Y}\|\leq \delta$.

Now pick a projection $P:E\to E$ with $P[E]=E\cap Y$ and norm $\|P\|\leq \dim(E)$, and define 
$S:E\to Y$ by $S=P+T(\mathrm{Id}-P)$. It is clear that 
$S(x)=x$ for every $x\in E\cap Y$. Furthermore, for every $x\in E$, the following relation holds:
\[\|Sx-Tx\| = \|(\mathrm{Id}-T)Px\|\leq \delta\|Px\|\leq \delta\dim(E)\|x\|<\frac{\epsilon}{2}\|x\|.\]
Therefore, for every $x\in E$ we have
\begin{align*}\|Sx\|&\leq \|Sx-Tx\|+\|Tx\|\leq \frac{\epsilon}{2}\|x\|+(1+\delta)\|x\|\leq (1+\epsilon)\|x\|.
\end{align*}
and
\begin{align*}
\|Sx\|&\geq \|Tx\|-\|Sx-Tx\| \geq (1-\delta)\|x\|-\frac{\epsilon}{2}\|x\|\geq \Big(1-\frac{\epsilon}{2}\cdot \frac{1-\epsilon}{1+\epsilon}-\frac{\epsilon}{2}\Big)\|x\|=(1+\epsilon)^{-1}\|x\|.
\end{align*}
which shows that $S$ is $(1+\varepsilon)$-isomorphic embedding and therefore completes the proof.
\end{proof}

The second main result is the following.

\begin{thm}\label{thm:richAiIdeals}Let $X$ be a Banach space. Then there exists a rectangular exceedingly rich family $\SSS$ of subspaces of $X\times X^*$ such that for every $V\times W\in \SSS$ there exists a norm-one operator $R:X\to V^{**}$ such that 
\begin{enumerate}[label=\textnormal{(R\alph*)}]
\item\label{list:Ra} $Rx=x$ for all $x\in V$. 
\item\label{list:Rb} $Rx(x^*|_{V})= x^*(x)$ for all $x\in X$ and for all $x^{*}\in W$.
\end{enumerate}
\end{thm}
\begin{proof}
By Theorem~\ref{thm:main3}, it suffices to find a rectangular family $\SSS$ large in the sense of Skolem-like functions such that for each $V\times W\in \SSS$ there exists a norm-one operator $R:X\to V^{**}$ satisfying both \ref{list:Ra} and \ref{list:Rb}.

Let $\phi:\PP(X\times X^*)\to \PP(X\times X^*)$ be the rectangular Skolem-like function from Key Lemma ~\ref{lem:keyLemma}. Pick $A\subset X\times X^*$ and put $V:=\overline{\phi_X(A)}$, $W:=\overline{\phi_{X^*}(A)}$. Let 
\[\begin{split}
\varGamma=\{(E,B,D,\epsilon)\colon & E\subset X \text{ finite-dimensional},\ B\in [E\cap \phi_X(A)]^{<\omega},\\
 & D\in [\phi_{X^*}(A)]^{<\omega}, \ \epsilon>0\}
\end{split}\]
be directed by the relation $\leq$ defined as follows 
$(E,B,D,\epsilon)\leq (E^\prime,B^\prime,D^\prime,\epsilon^\prime)$ if and only if $E\subset E^\prime$, $B\subset B^\prime$, $D\subset D^\prime$ and $\epsilon^\prime \leq \epsilon$.

Since $\phi$ is as in Key Lemma~\ref{lem:keyLemma}, for every $I=(E^\prime,B^\prime,D^\prime,\epsilon^\prime)$ there exists a bounded linear operator
$T_I:E\to V$ satisfying conditions \ref{list:keyLemma1}, \ref{list:keyLemma2} and \ref{list:keyLemma3} of the lemma.

We define $R_I:X\to V$ by 
\begin{displaymath}
R_Ix= \left\{
\begin{array}{ll}
T_Ix & \text{ if }x\in E;\\
0 & \text{ otherwise }.
\end{array} \right.
\end{displaymath}
 
We note that for every $x\in X$ and $x^*\in X^*$, $\{x^*(R_Ix):I\in \varGamma\}$ is a bounded set of real numbers and so relatively compact. We fix a directed ultrafilter $\mathcal{U}$ in $\varGamma$ (that is, $\U$ is nonprincipal ultrafilter satisfying $\{i\in\Gamma\colon i\geq i_0\}\in \U$ for every $i_0\in\Gamma$) and define a function $R:X\to V^{**}$ by the formula:
\[Rx(x^*)=\lim_{\mathcal{U}}x^*(R_Ix),\qquad x\in X,\; x^*\in V^*.\]

It is readily seen that $R$ is a linear map. Moreover, for arbitrary $x \in B_X$, $x^*\in B_{V^*}$ and $\epsilon>0$, if $I_0=(\mathrm{span}\{x\},\emptyset,\emptyset,\epsilon)$ then for every $I\geq I_0$ we have $|x^*(T_Ix)|\leq \|T_I\|\leq (1+\epsilon)$ and therefore
\[|Rx(x^*)|=\lim_{\mathcal{U}}|x^*(R_Ix)|\leq (1+\epsilon)\]
and we deduce that $\|R\|\leq 1$.

To verify \ref{list:Ra}, we let $x \in \phi_X(A)$ be arbitrary and set $I_0 = (\mathrm{span}\{x\}, \{x\}, \emptyset, 1)$. Then, for any $I \geq I_0$, recalling \ref{list:keyLemma1}, for all $x^* \in V^*$, we have
\[x^*(x)=x^*(T_Ix)=x^*(R_Ix),\]
which implies
\[Rx(x^*) = \lim_{\mathcal{U}} x^*(R_Ix)=\lim_{I\geq I_0} x^*(R_Ix)=x^*(x).\]
Thus, we have $Rx=x$ for $x\in \phi_X(A)$ and by the continuity of $R$ we obtain \ref{list:Ra}.

To verify \ref{list:Rb}, we let $x\in X$ and $x^*\in \phi_{X^*}(A)$ be arbitrary. Given $\delta>0$ we set $I_0=(\mathrm{span}\{x\},\emptyset, \{x^*\}, \delta/(1+\|x^*\|\|x\|))$. For every $I=(E,B,D,\epsilon)\geq I_0$, according to \ref{list:keyLemma3}, we have

\[|(x^*|_{V})(T_Ix)-x^*(x)| \leq \|(x^*\circ T_I-x^*)|_E\| \|x\|\leq \epsilon\|x^*\|\|x\|<\delta.\] 
This implies
\begin{align*}
|Rx(x^*|_{V})- x^*(x)|&= \lim_{\mathcal{U}}|(x^*|_{V})(T_Ix)-x^*(x)|\leq \delta.
\end{align*}
As $\delta>0$ is arbitrary, we can conclude that $Rx(x^*|_{X_M})=x^*(x)$. Furthermore, due to the continuity of $R$, we can deduce that $Rx(x^*|_{V})=x^*(x)$ holds for all $x\in X$ and for all $x^{*}\in \overline{\phi_{X^*}(A)} = W$.
\end{proof}

\begin{cor}\label{cor:projectionAiIdeal}Let $X$ be a Banach space. Then there exists a rectangular exceedingly rich family $\SSS$ of subspaces of $X\times X^*$ such that for every $V\times W\in \SSS$ there exists a norm-one projection $P:X^*\to X^*$ with $\ker P = V^\perp$ and $\rng P\supset W$.
\end{cor}
\begin{proof}Let $\SSS$ be as in Theorem~\ref{thm:richAiIdeals}. Then given $V\times W\in \SSS$ there exists norm-one mapping $R:X\to V^{**}$ such that \ref{list:Ra} and \ref{list:Rb} hold. Define $P:X^*\to X^*$ by 
\begin{equation}\label{eq:formulaProjection}
Px^*(x):=Rx(x^*|_{V}),\qquad x\in X,\; x^*\in X^*.
\end{equation}
Then indeed $P$ is a norm-one operator and using condition \ref{list:Ra} we obtain $Px^*|_V = x^*|_V$, $x^*\in X^*$ which implies that $P$ is a projection. Condition \ref{list:Rb} implies that $\rng P\supset W$. It is immediate from \eqref{eq:formulaProjection} that $Px^* = 0$ for $x^*\in V^\perp$, so we have $V^\perp \subset \ker P$. Finally, given $x^*\in \ker P$ we have $x^*|_V = Px^*|_V \equiv 0$ and so $\ker P\subset V^\perp$.
\end{proof}

One can wonder what is the range of the projection given by Corollary~\ref{cor:projectionAiIdeal}. Theorem~\ref{thm:asplundAndWLD} below shows that this issue might be quite complicated and might depend on the structure of the Banach space $X$ (as one can see from the proof, this result is essentially known, we only use outcomes of Section~\ref{sec:formulations} to reformulate everything using rich families instead of suitable models).

\begin{thm}\label{thm:asplundAndWLD}
Let $X$ be a Banach space. Consider the following conditions.
\begin{enumerate}[label=\textnormal{(\roman*)}]
    \item\label{it:asplund} There exists a rectangular exceedingly rich family $\SSS$ of subspaces of $X\times X^*$ such that for every $V\times W\in \SSS$ there exists a norm-one projection $P:X^*\to X^*$ with $\ker P = V^\perp$ and $\rng P = W$.
    \item\label{it:wld} There exists a rectangular exceedingly rich family $\SSS$ of subspaces of $X\times X^*$ such that for every $V\times W\in \SSS$ there exists a norm-one projection $P:X^*\to X^*$ with $\ker P = V^\perp$ and $\rng P \supset \overline{W}^{w^*}$.
    \item\label{it:wldII} There exists a rectangular exceedingly rich family $\SSS$ of subspaces of $X\times X^*$ such that for every $V\times W\in \SSS$ there exists a $w^*$-$w^*$ continuous norm-one projection $P:X^*\to X^*$ with $\ker P = V^\perp$ and $\rng P \supset W$.
\end{enumerate}
Then
\begin{itemize}
    \item \ref{it:asplund} holds $\Leftrightarrow$ $X$ is Asplund,
    \item \ref{it:wld} holds $\Leftrightarrow$ \ref{it:wldII} holds $\Leftrightarrow$ $X$ is WLD.
\end{itemize}
\end{thm}
\begin{proof}
By \cite[Proposition 26]{KubisSkeleton}, $X$ is Asplund if and only if $X$ is a subset of the set induced by a projectional skeleton on $X^*$. Thus, by Lemma~\ref{lem:skeleton} (applied to $X\times D = X^*\times X$) we obtain that \ref{it:asplund} holds if and only if $X$ is Asplund.

It is immediate that \ref{it:wldII} implies \ref{it:wld}.

By \cite[Corollary 25]{KubisSkeleton}, $X$ is WLD if and only if $X^*$ is the set induced by a projectional skeleton on $X$.

First, suppose that $X$ is WLD. By Lemma~\ref{lem:skeleton}, there exists a rectangular exceedingly rich family $\SSS$ of subspaces of $X\times X^*$ such that for every $V\times W\in \SSS$ there exists a projection $Q:X\to X$ with $Q[X] = V$ and $\ker Q = W_\perp$. Then $P = Q^*:X^*\to X^*$ is a $w^*$-$w^*$ continuous projection with $\rng P = \overline{W}^{w^*}$ and $\ker P = V^\perp$. Thus, if $X$ is WLD then both \ref{it:wld} and \ref{it:wldII} hold.

Finally, if \ref{it:wld} holds, then for every $V\times W\in\SSS$ we have $X = \overline{V + W_\perp}$ (because otherwise there is $x^*\in X^*\setminus\{0\}$ such that $x^*|_{V + W_\perp}\equiv 0$, and therefore $x^*\in V^\perp\cap (W_\perp)^\perp\subset \rng P\cap \ker P = \{0\}$, contradiction). Thus, by Lemma~\ref{lem:skeleton} we obtain that $X^*$ is the set induced by a projectional skeleton on $X$ and so $X$ is WLD.
\end{proof}

\section{Almost isometric local retracts}\label{sec:aiLocRetr}

The main aim of this section is to show that small modifications of the proof of Theorem~\ref{thm:richAiIdeals} lead quite directly to the proof of Theorem~\ref{thm:main2}. This gives a different approach when compared to the proof of \cite[Theorem 5.5]{QuilisZoca}, where a similar (but slightly weaker) result was proved.

\begin{lemma}[Key Lemma, metric case]\label{lem:keyLemmaMetric}
	Let $X$ be a metric space. Then there exists a rectangular Skolem-like function $\phi:\PP(X\times \Lip_0(X))\to \PP(X\times \Lip_0(X))$ such that for every $A\subset X\times \Lip_0(X)$ the set $\phi_{\Lip_0(X)}(A)$ is $\Rat$-linear subspace and the following holds :
 
    Let $\varepsilon > 0$, $D \in [\phi_{\Lip_0(X)}(A)]^{<\omega}$, $B \in [\phi_X(A)]^{<\omega}$ and $E\subset X$ be a finite set with $E \supset B$. Then there is a Lipschitz mapping $T: E \to \phi_X(A)$ such that
	\begin{enumerate}[label=\textnormal{(K'\alph*)}]
		\item\label{list:keyLemma1Metric} $Tx = x$ for all $x\in B$,
		\item\label{list:keyLemma2Metric} $(1+\varepsilon)^{-1}d(x,y) \leq d(Tx,Ty) \leq (1+\varepsilon) d(x,y)$  for all $x,y \in E$,
		\item\label{list:keyLemma3Metric}$\|(d\circ T - d)|_E\|\leq \varepsilon \|d\|$ for all $d\in D$,
	\end{enumerate}
\end{lemma}
\begin{proof}The proof is very similar to the proof of Key Lemma~\ref{lem:keyLemma}. The additional ingredient is only to use Lipschitz-free spaces. We pick arbitrary $0\in X$, so $X$ is a pointed metric space.

Given sets $A\subset X\times \Lip_0(X)$, finite sets $B\subset X$ and $D\subset \Lip_0(X)$, for each $n,k \in \Nat$ and $\delta\in \Rat_+$, we pick a finite $\delta$-net $A_{k,\delta}$ of $S_{\ell_1^k}$ and a finite $\delta$-net $N_{n,\delta}(B)\subset \F(B) := \Span \delta(B) (\subset \F(X))$ of $nB_{\F(B)}$. Now since for any $R\in\Nat$ the mapping $\psi_{n,k,\delta,R}: (B_X(0,R))^k \to (\mathbb{R}^{N_{n,\delta}(B) \times A_{k,\delta}} \times \mathbb{R}^{D \times k},\|\cdot\|_\infty)$, defined by the formula
\[
		\psi_{n,k,\delta,R}(u) := \left(\left(\left\|x + \sum_{i=1}^{k} a_i\delta(u_i) \right\|\right)_{x\in N_{n,\delta}(B),a\in A_{k,\delta}},\left(d(u_i)\right)_{d\in D, i\leq k}\right),
\]
has separable range, we may pick a countable set $F_{n,k,\delta,R}(B,D)\subset B_X(0,R)$ such that $\psi_{n,k,\delta,R}(F_{n,k,\delta,R}(B,D)^k)$ is dense in the range of $\psi_{n,k,\delta,R}$.
Next, if $D(A):=\Span_\Rat \pi_{\Lip_0(X)}(A)\subset X^*$ for every $A\subset X\times \Lip_0(X)$, we inductively define $\xi(A)\subset X$ for $A\subset X\times \Lip_0(X)$ as follows:
\[\begin{split}
\xi_0(A)&:=\pi_X(A),\\
\xi_{n+1}(A)&:=\xi_n(A)\cup \bigcup\Big\{F_{n,k,\delta,R}(B,D)\colon n,k,R\in\Nat,\; \delta\in\Rat_+,\\
&\qquad\qquad B\subset \xi_n(A) \text{ and } D\subset D(A) \text{ are finite sets}\Big\},\\
\xi(A)&:=\bigcup_{n=0}^\infty \xi_n(A).
\end{split}\]

Similarly as in the proof of Key Lemma~\ref{lem:keyLemma} we check that indeed the mapping $\phi$ defined by $\phi(A):=\xi(A)\times D(A)$, $A\subset X\times \Lip_0(X)$ is a rectangular Skolem-like function.

Now, pick $\varepsilon>0$, $A\subset X\times X^*$, $D\subset D(A) = \phi_{\Lip_0(X)}(A)$ finite, $B\subset \phi_X(A) = \xi(A)$ finite and finite set $E\subset X$ with $E\supset B$. Let $R\in\Nat$ be such that $E\subset B(0,R)$ and enumerate $\{e_1,\ldots,e_k\} = E\setminus B$. Let $\zeta$ be the function from Lemma~\ref{lem:calculations} applied to $m=R$ and $\F(B)\oplus \F(E\setminus B)$ and let $\delta>0$ be such that $\zeta(\delta)<\varepsilon$. Now, pick $n\in \Nat$ with $n\geq \tfrac{1}{\delta}$. By the choice of $F_{n,k,\delta,R}(B,D)$, there are $f_1,\ldots,f_k\in F_{n,k,\delta,R}(B,D)\subset \xi(A)$ such that $\|\psi_{n,k,\delta,R}((f_i)) - \psi_{n,k,\delta,R}((e_i))\|_\infty$. Then, by Lemma~\ref{lem:calculations}, since $\zeta(\delta)<\varepsilon$, we have that the mapping $\widehat{T}:\F(E) = \F(B)\oplus \F(E\setminus B)\to \F(\overline{\xi(A)}) = \F(\overline{\phi_X(A)})$ given by
 \[
 \widehat{T}(x+\sum_{i=1}^{k}a_i \delta(e_i)):=x+\sum_{i=1}^{k}a_i \delta(f_i),\quad x\in \F(B),\; a\in \Rea^{k}
 \]
 is $(1+\varepsilon)$-isomorphic embedding. Thus, the mapping $T:E\to \phi_X(A)$ given by $T(x):=\delta^{-1}(\widehat{T}(\delta(x)))$ satisfies \ref{list:keyLemma1Metric} and \ref{list:keyLemma2Metric}. Finally, since for any $d\in D$ we have $|d(e_i)-d(f_i)|<\delta$, $i=1,\ldots,k$, the moreover part of Lemma~\ref{lem:calculations} implies that \ref{list:keyLemma3Metric} holds as well.
\end{proof}

\begin{remark}In the proof of Lemma~\ref{lem:keyLemmaMetric} we used Lemma~\ref{lem:calculations}. We could also formulate and prove a metric version of Lemma~\ref{lem:calculations}, whose proof is much easier. This is the basic strategy of the approach from \cite{QuilisZoca} (see the proof of Lemma 5.2 therein). In order to shorten a bit the argument, we used rather the already proven Lemma~\ref{lem:calculations} (even though an easier statement would be sufficient as well in this case).
\end{remark}

\begin{thm}\label{Thm:ExeedRichFamilyMetric}Let $X$ be a metric space. Then there exists an exceedingly rich family $\SSS$ of closed subsets of $X$ such that each $F\in \SSS$ is ai local retract in $X$.
\end{thm}
\begin{proof}By Theorem~\ref{thm:main3}, it suffices to find a rectangular family $\SSS$ large in the sense of Skolem-like functions such that each $F\in \SSS$ is almost isometric local retract in $X$.

Pick the rectangular Skolem-like function $\phi:\PP(X\times \Lip_0(X))\to \PP(X\times \Lip_0(X))$ from the statement of Lemma~\ref{lem:keyLemmaMetric} and consider its projection onto the first coordinate given by $\psi(A):=\phi_X(A\times \{0\})$. It is easy to check that $\psi:\PP(X)\to \PP(X)$ is Skolem-function as well. Thus, in order to finish the proof it suffices to check that $\overline{\psi(A)}$ is almost isometric local retract for every $A\subset X$.

Pick $A\subset X$ and denote $V:=\overline{\psi(A)}$. Fix $\varepsilon>0$ and a finite set $E\subset X$. Pick $\delta>0$ small enough (more concretely, $\delta < \varepsilon/2$ and $\delta<\min\{\tfrac{\varepsilon d(x,y)}{2}\colon x,y\in E, x\neq y\}$) and a function $\varphi:E\cap V \to \psi(A)$ such that $d(\varphi(x),x)<\delta$ for every $x\in E\cap V$. We let $E' = E\cup \rng \varphi$, $B=\rng \varphi$, $D= \emptyset$ and use \ref{list:keyLemma1Metric} and \ref{list:keyLemma2Metric} to find $(1+\delta)$-biLipschitz embedding $T:E'\to \psi(A)$ satisfying $Tx=x$ for $x\in B$.

We define $S:E\to V$ by $Sx = x$ for $x\in E\cap V$ and $Sx = T(x)$ for $x\in E\setminus V$. It suffices to show that $S$ is $(1+\varepsilon)$-biLipschitz embedding. Since $T$ is $(1+\varepsilon)$-biLipschitz embedding, it suffices to consider $x\in E\cap V$, $y\in E\setminus V$ and estimate then $d(Sx,Sy)$. For $x\in E\cap V$, $y\in E\setminus V$ we have
\[\begin{split}
|d(Sx,Sy) - d(x,y)| & = |d(x,Ty) - d(x,y)|\leq \delta +  |d(Ty,\varphi(x)) - d(x,y)|\\
& = \delta +  |d(Ty,T\varphi(x)) - d(x,y)|\leq \delta + \delta d(x,y)\leq \varepsilon d(x,y).
\end{split}\]
Thus, $S$ is $(1+\varepsilon)$-biLipschitz embedding.
\end{proof}

Now, we could also proceed similarly as in the proofs of Theorem~\ref{thm:richAiIdeals} and Corollary~\ref{cor:projectionAiIdeal} and obtain a family of extension operators $E:\Lip_0(V)\to \Lip_0(X)$; however, in this case it seems to us that we would not get anything valuable since the existence of an extension operator $E:\Lip_0(V)\to \Lip_0(X)$ follows immediately from the existence of Hahn-Banach extension operator $H:\F(V)^*\to \F(X)^*$ which in turn is easily deduced from Theorem~\ref{thm:richAiIdeals}. In other words, using directly Corollary~\ref{cor:projectionAiIdeal} we obtain a family of projection on the space $\Lip_0(X) = \F(X)^*$.

Finally, let us recall that any nonseparable Lipschitz-free space contains an isomorphic copy of $\ell_1(\omega_1)$, see \cite[Theorem 2.1]{Kal11} or \cite[Proposition 3]{HajekNovotny} and also \cite[Theorem 3.9]{AACD20} for even more general statements, and therefore, nonseparable Lipschitz-free spaces are not WLD and also not Asplund. In particular, projections on $\Lip_0(X)$ given by Corollary~\ref{cor:projectionAiIdeal} are not $w^*$-$w^*$ continuous due to Theorem~\ref{thm:asplundAndWLD}.

\section{(Separable) determination theorems}\label{sec:appl}

In this section we exhibit the use of our results from the previous sections to obtain several ``(separable) determination theorem''. The main advantage is that the notion of exceedingly rich families makes it possible to combine several results together, see Lemma~\ref{lem:combineRich}. This enables us to prove independently some partial steps and then combine those together without the need of going through the proofs again. We refer the interested reader to \cite[Introduction]{C18}, where this is explained in a greater detail. The main results from this section are Theorem~\ref{thm:gurarii}, Theorem~\ref{thm:daugavet} and Theorem~\ref{thm:urysohn} from which using Lemma~\ref{lem:combineRich} we easily obtain Theorem~\ref{thm:main4}.

\begin{remark}\label{rem:moreApplications}
Results contained in this section should be understood just as a sample of possible applications as using similar techniques we are able to add several other properties to the list from Theorem~\ref{thm:main4} such as local diameter 2 property, diameter 2 property, strong diameter 2 property, almost squreness, octahedrality etc (we refer the interested reader to \cite{Abr15} where the notions mentioned above are defined and results of a similar nature were proved therein). Similarly, one could add some more properties of metric spaces to the statement of Theorem~\ref{thm:main5} such as being length, having long trapezoid property etc (we refer the interested reader to \cite{GPZ} and \cite{PZ}, where the connection to the above mentioned Daugavet property and octahedrality is proved).
\end{remark}

We recall that a Banach space $X$ is said to be a \emph{Gurari\u{\i} space} (an \emph{$L_1$-predual space}) if, for every $\epsilon > 0$ and for every isometric embedding $T : E \to X$ and for all finite-dimensional space $F$, if $E\subset F$, then $T$ admits an extension $\hat{T}: E \to  X$ such that  $(1+\epsilon)^{-1}\|x\|\leq \|\hat{T}\|\leq (1+\epsilon)\|x\|$ ($\|\hat{T}\|\leq (1+\epsilon)\|x\|$) for every $x\in F$ (see \cite[Proposition 2.7]{Ban18}) .  In other words, $X$ is a Gurari\u{\i} space (an $L_1$-predual space) if and only if it is an ai-ideal (a ideal) in every Banach space containing it.

The subsequent lemma is a slight modification of \cite[Proposition 2.5]{Ban18}. We provide the complete argument here for the reader's convenience.

\begin{lemma}\label{lem:inheritGurarii}
Let $X$ be a Gurari\u{\i} space (an $L_1$-predual space). If $Y$ is an ai-ideal on $X$, then
$Y$ is also a Gurari\u{\i} space (an $L_1$-predual space).
\end{lemma}
\begin{proof}
Given $\epsilon > 0$ and a finite dimensional Banach space $E$ we let $T:E\to Y$ be a linear isometric embedding. If $F$ denotes a finite dimensional Banach containing $E$, since $X$ is a Gurari\u{\i} space (an $L_1$-predual space), for $0<\delta<\sqrt{1+\epsilon}-1$, $T$ admits a linear extension $\hat{T}: F\to X$ of $T$ with $(1+\delta)^{-1}\|x\| \leq \|T(x)\|\leq (1 + \delta)\|x\|$ ($\|T(x)\|\leq (1 + \delta)\|x\|$) for all $x\in E$. Since $Y$ is an ai-ideal in $X$, there exists an $(1+\delta)$-isomorphic embedding $S:\hat{T}[F] \to Y$ such that $S(x)=x$ for every $x\in Y\cap \hat{T}[F]$. Hence $S\circ \hat{T}: F \to Y$ is a linear extension of $T$ satisfying $(1 + \epsilon)^{-1}\|x\|\leq (1 + \delta)^{-2}\|x\|\leq \|S\circ T(x)\| \leq (1 + \delta)^2\|x\|\leq (1 + \epsilon)\|x\|$ ($\|S\circ T(x)\| \leq (1 + \delta)^2\|x\|\leq (1 + \epsilon)\|x\|$) for all $x\in F$. 
\end{proof}

From the results established in this paper, we can obtain a slight strengthening of \cite[Theorem 3.4]{GarbuKubis} and \cite[Theorem 2.6 and Theorem 2.8]{Ban18}.

\begin{thm}\label{thm:gurarii}
Let $X$ be a non-separable Banach space. Then there exists an exceedingly rich family $\SSS$ of subspaces of $X$ such that, for every $Y\in \SSS$ the following holds
\[\begin{split}
X \text{ is a Gurari\u{\i} space}  & \Leftrightarrow Y \text{ is a Gurari\u{\i} space},\\
X \text{ is an }L_1\text{-predual space}  & \Leftrightarrow Y \text{ is an } L_1\text{-predual space.}
\end{split}\]
\end{thm}
\begin{proof}
We start with handling the first equivalence by distiguishing two cases. In the case that $X$ is Gurari\u{\i} space, by Theorem~\ref{thm:richAiIdeals} there is exceedingly rich family $\SSS_1$ such that every $Y\in\SSS_1$ is ai-ideal in $X$ and therefore, by Lemma~\ref{lem:inheritGurarii}, it is Gurari\u{\i} space. Thus, we have that for any $Y\in\SSS_1$ it holds that $X$ is Gurari\u{\i} space if and only if $Y$ is Gurari\u{\i} space.

In the case that $X$ is not Gurari\u{\i} space, there is $\epsilon>0$, finite dimensional $E$ and $F$, with $E\subset F$, and an isometric linear embedding $T:E\to X$ that cannot be extended to an $(1+\epsilon)$-isomorphic embedding from $F$ into $X$. Then we let 
\[\SSS_1:=\{V\subset X\colon V\text{ is a closed subspace of }X \text{ and contains } E\cup F\cup T(E)\}\]
and easily check that $\SSS_1$ is exceedingly rich family with $V$ not being Gurari\u{\i} for $V\in \SSS_1$. Thus, we have that for any $Y\in\SSS$ it holds that $X$ is not Gurari\u{\i} if and only if $Y$ is not Gurari\u{\i}.

Thus, in any case there exists exceedingly rich family $\SSS_1$ such that for any $Y\in\SSS_1$ we have that $X$ is Gurari\u{\i} if and only if $Y$ is Gurari\u{\i}. Similarly, we find exceedingly rich family $\SSS_2$ such that for any $Y\in\SSS_2$ we have that $X$ is $L_1$-predual if and only if $Y$ is $L_1$-predual. Finally it suffices to put $\SSS:=\SSS_1\cap \SSS_2$, which is exceedingly rich family by Lemma~\ref{lem:combineRich}.
\end{proof}

A Banach space $X$ is said to have the \emph{Daugavet property} if for every rank-one operator $T:X\to X$ satisfies  
$\|\Id_X-T\|=1+\|T\|$. It is well-known, see \cite[Lemma 2.2]{Kadets2000}, that $X$ has the Daugavet property if and only if, for every $y \in S_X$, $x^* \in S_{X^*}$ and $\epsilon > 0$ there exists $x \in S_X$ such that
$x^*(x) \geq 1-\epsilon$ and $\|x + y\| \geq 2 - \epsilon$.

By employing the results from this paper, we establish the following.

\begin{thm}\label{thm:daugavet}
Let $X$ be a non-separable Banach space. Then there exists an exceedingly rich family $\SSS$ of subspaces of $X$ such that, for every $Y\in \SSS$ the following holds
\[
X \text{ has the Daugavet property} \Leftrightarrow Y\text{ has the Daugavet property}.
\]
\end{thm}
\begin{proof}Similarly as in the proof of Theorem~\ref{thm:gurarii}, we distiguish two cases.
If $X$ has the Daugavet property, then according to Theorem~\ref{thm:richAiIdeals} there exists 
an exceedingly rich family $\SSS$ such that every $Y\in\SSS$ is ai-ideal in $X$, and, by \cite[Proposition 3.8]{Abr14}, each of these ai-ideals inherits the Daugavet property. Thus, if $X$ has the Daugavet property, $\SSS$ does the job.

If $X$ does not have Daugavet property, then, based on the characterization mentioned above, there exist $y\in S_X$, $x^*\in S_{X^*}$ and $\epsilon>0$ such that, for every $x\in S_X$, $x^*(x)<1-\epsilon$ or $\|x + y\|<2 - \epsilon$.

We let $D\subset B_X$ be a countable set witnessing that $\|x^*\|=1$, that is, $\sup_{x\in D}\|x^*(x)\|=1$. The family $\SSS':=\{Y\subset X\colon $Y$\text{ is a closed subspace of }X\text{ and contains } \{y\}\cup D\}$ is readily seen to be an exceedingly rich family of closed subsets of $X$. Moreover, for each $Y\in \SSS'$ we have $\|x^*|_Y\|=1$, and since 
$S_Y\subset S_X$, for each $x\in S_Y$ we have $x^*(x)<1-\epsilon$ or $\|x + y\|<2 - \epsilon$, which proves that $Y$
does not have the Daugavet property. Thus, if $X$ does not have the Daugavet property, $\SSS'$ does the job.
\end{proof}

A metric space $X$ is said to be an \emph{absolute ai-local retract} if it is an ai-local retract in every metric space containing it. This concept was introduced in \cite[Definition 4.1]{QuilisZoca}. It was also established, see \cite[Theorem 4.5]{QuilisZoca}, that $X$ absolute ai-local retract, if and only if, for every finite subsets $E$, $F\subset X$, with $E\subset F$, every isometry $f:E\to X$ can be extended to an $(1+\epsilon)$-biLipschitz embedding $\widehat{f}:F\to X$.

From the results of the Section~\ref{sec:aiLocRetr}, we have the following application.

\begin{thm}\label{thm:urysohn}
Let $X$ be a non-separable metric space. Then there exists an exceedingly rich family $\SSS$ of subspaces of $X$ such that, for every $Y\in \SSS$ the following holds
\[
X \text{ is an absolute ai-local retract} \Leftrightarrow Y\text{ absolute an ai-local retract}.
\]
\end{thm}
\begin{proof}
Similarly as in the proof of Theorem~\ref{thm:gurarii} by distinguishing two cases. If $X$ is an absolute ai-local retract, then from Theorem~\ref{Thm:ExeedRichFamilyMetric} we may fix a exceedingly rich family $\SSS$ comprising ai-local retracts of $X$. With an argument similar to Lemma~\ref{lem:inheritGurarii} we deduce that each $F\in \SSS$ is absolute ai-local retract.
On the other hand, assuming that $X$ is not an absolute ai-local retract, there exists $\epsilon>0$, finite sets $E$, $F\subset X$ and a isometric mapping $f:E:\to X$ that cannot be extended to a $(1+\epsilon)$-biLipschitz embedding. The family $\SSS=\{Y\subset X:Y\text{ contains }E\cup F\cup T(E)\}$ is clearly a exceedingly rich family 
such that $Y$ is not an absolute ai-local retract for all $Y\in \SSS$.
\end{proof}

\noindent{\bf Acknowledgements.}
We would like to thank Andr\'es Quilis and Abraham Rueda Zoca for sending us preliminary version of their preprint \cite{QuilisZoca}, which motivated us to write this paper. M. C\'uth would like to thank Universidade Federal de S\~{a}o Paulo - UNIFESP for their hospitality and generosity during his visit in November 2023, when a lot of the work on this paper was undertaken (part of the visit was covered by the grant PROAP/CAPES
no. 23089.032403/2023-95). Results from Section~\ref{sec:aiIdeals} were inspired by the thesis \cite{Smetana}, which was written by the third author under the supervision of the second author.

\end{document}